\documentclass{amsart}
\usepackage[utf8]{inputenc}
\usepackage{comment}
\usepackage{url}

\newtheorem{theorem}{Theorem}[section]
\newtheorem{lemma}{Lemma}[section]

\title{An algorithm and computation to verify Legendre's Conjecture up to  $3.33\cdot10^{13}$}

\author{Jonathan Sorenson}
\author{Jonathan Webster}
\address[1]{Computer Science and Software Engineering Department, Butler University, Indianapolis, IN USA, \texttt{jsorenso@butler.edu}}
\address[2]{Mathematical Sciences Department, Butler University, Indianapolis, IN USA, \texttt{jewebste@butler.edu}}

\begin{document}

\maketitle

\begin{abstract}
We state a general purpose algorithm for quickly finding primes in evenly divided sub-intervals. Legendre's conjecture claims that for every positive integer $n$, there exists a prime between $n^2$ and $(n+1)^2$.  Oppermann's conjecture subsumes Legendre's conjecture by claiming there are primes between $n^2$ and $n(n+1)$ and also between $n(n+1)$ and $(n+1)^2$.  Using Cram\'er's  conjecture as the basis for a heuristic run-time analysis, we show that our algorithm can verify Oppermann's conjecture, and hence also Legendre's conjecture, for all $n\le N$ in time $O( N \log N \log \log N)$ and space $N^{O(1/\log \log N)}$.  We implemented a parallel version of our algorithm and improved the empirical verification of Oppermann's conjecture from the previous $N = 2\cdot 10^{9}$ up to $N = 3.33\cdot 10^{13}$, so we were finding $27$ digit primes.  The computation ran for about half a year on four Intel Xeon Phi $7210$ processors using a total of $256$ cores.
\end{abstract}

\section{Introduction and Motivation}

In his work on number theory, Adrien-Marie Legendre conjectured 
that for every positive integer $n$ there is always a prime between $n^2$ and $(n+1)^2$.  In 1882, the Danish mathematician Ludwig Oppermann strengthened the conjecture.  He posited the existence of two primes between each of $n^2$, $n(n+1)$, and $(n+1)^2$.   

While it is commonly believed that these conjectures are true, we remain far from proving them.  After all, Legendre's conjecture implies a gap between primes $p$ that is of size $O(\sqrt{p})$.  Even the Riemann hypothesis would only give $O(\sqrt{p} \log{p})$.  

Our goal is to create an algorithm to computationally verify Oppermann and Legendre's conjectures as far as possible.  The previous record for this is $n\le 2\times 10^9$, as a consequence of a massive computation that found all primes up to $4\times10^{18}$ due to \cite{SHP2014}.
We have extended this to $n\le 3\times 10^{13}$ as of this writing.


While the correctness of our algorithms will be proved unconditionally, we need conjectures even stronger than Legendre's conjecture to provide any meaningful asymptotic analysis of the algorithm.  We will rely on the probabilistic model that undergirds Cram\'er's conjecture.  Namely, that if you choose a random number less than $x$, then it is expected to be prime with probability $1/\log x$.  Throughout, we assume Oppermann's conjecture is true.  Of course it may not be true, and our algorithms will detect a counterexample interval if there is one below $N$, but for the purposes of asymptotic runtime analysis we assume this will not happen.  We present algorithms to verify an unproven conjecture; if the conjecture is false, no such algorithm can exist.

The organization of the paper is as follows.  
In Section 2, we give some background on prime gaps and state the heuristic model we will use for asymptotic analysis.  
Since we believe that this is the first computation specifically aimed at empirical verification of Legendre's conjecture, we use Section 3 to discuss three increasingly better but more complicated algorithms that provide context for our work.   
In Section 4, we state our algorithm and Section 5 provides the asymptotic analysis of the running time and space used by our algorithm.  
We conclude with some timing information, details of our computation, and comments about the implementation in the final section.   

\section{Prime Gaps}

We know from the prime number theorem that $\pi(x)$, the number of primes $\le x$, is asymptotically $x/\log x$.  If we were to choose an integer uniformly at random below $x$, then the probability it is prime is asymptotic to $1/\log x$ as a result, and further, we assume independence.  Of course, primes are not random, but using this idea to predict the distribution of primes under various conditions has proven to be very useful.
Indeed, Cram\'er conjectured that the maximum gap between consecutive primes near $n$ is $O((\log n)^2)$ based on this model, and as a result we refer to this as \textbf{Cram\'er's Model}.
The massive computation from \cite{SHP2014} has shown the model works well in practice.  See \cite{FGKT2016} for a discussion of the model and references to other work on prime gaps and related conjectures.  See also \cite[\S1.4]{CP}.


We state a couple simple but useful consequences.
\begin{lemma}\label{cramer1}
Assuming Cram\'er's model, the probability that all integers from a set of size $\log n \log v$ near $n$ are composite is at most $O(1/v)$, for large $n$. 
\end{lemma}
\begin{proof}
    Assuming independence, this probability is 
    $$
    \left( 1- \frac{1}{\log n}\right)^{\log n\log v}
    \quad\sim\quad e^{-\log v} \quad=\quad \frac{1}{v}.
    $$
\end{proof}
Setting $v=n$ gives Cram\'er's conjecture.

Let $M$ be a positive integer.
Next, we apply Cram\'er's model to primes in arithmetic progressions.
By Dirichlet's theorem \cite{apostol2} we know that, asymptotically, there are $x/(\phi(M)\log x)$ primes $\le x$ in each of the $\phi(M)$ residue classes $a$ modulo $M$ where $\gcd(a,M)=1$.  Here $\phi(M)$ is Euler's totient function.  This means an integer chosen uniformly at random from a fixed residue class modulo $M$, for example $1\bmod M$, is prime with probability $M/(\phi(M)\log x)$.
\begin{lemma}\label{cramer2}
For positive integers $n$, $b$, and $v$,
  let $M$ be a positive integer that is a multiple of all primes $p\le b$.
Assuming Cram\'er's model, the probability that $\log n \log v / \log b$ integers near $n$ that are all relatively prime to $M$ are all composite is at most $O(1/v)$, for large $n$ and $b$ tending to infinity as a function of $n$.
\end{lemma}

\begin{proof}
WLOG we can assume all the integers are in the same residue class modulo $M$.
    Let $m=\prod_{p\le b} p$, which divides $M$.
    Then we have 
    \begin{eqnarray*}
    \frac{M}{\phi(M)}\frac{1}{\log n}& \ge& \frac{m}{\phi(m)}\frac{1}{\log n}\\
    &=& \frac{1}{\log n}\prod_{p\le b} \frac{p}{p-1} \\
    &=& \frac{1}{\log n}\left(\prod_{p\le b} 1-\frac{1}{p}\right)^{-1} \\
    &\sim& \frac{1}{\log n} e^\gamma \log b
    \end{eqnarray*}
    by Mertens's theorem \cite{HW}.
    Following the lines of the previous lemma, the result follows.
\end{proof}

\section{Three Algorithms}

Of course we can do better than finding all primes up to $N^2$, which is what was done in \cite{SHP2014}.  In this section we outline three increasingly better, but increasingly more complicated, ways to do this that illustrate the design choices for our new algorithm, which is described in detail in the next section.

\subsection{Algorithm A}
One way to verify Oppermann's conjecture for all $n\le N$
would be, for each $n$, to
test $n^2+1, n^2+2, \ldots$ for primality until we find a prime.
Then test $n(n+1)+1, n(n+1)+2, \ldots$ and do the same.
We can employ some of the techniques outlined in \cite{Sorenson06} to make this efficient:
\begin{itemize}
\item We can do a small bit of trial division and use base-2 strong pseudoprime tests to quickly discard composites.
We can easily do this so that the average time spent on each composite is at most $O(\log n)$.
\item We need an unconditional, fast proof of primality.  If we are willing to assume the ERH for runtime analysis (not correctness) the pseudosquares prime test \cite{LPW96} takes $O((\log n)^3)$ arithmetic operations.  Otherwise we can use AKS \cite{AKS04,Bernstein2004}, in $O((\log n)^4)$ time.
\end{itemize}
By Cram\'er's model, we expect to find each desired prime after ruling out $O(\log n)$ composites on average.  Since we are finding $2N$ primes, the overall runtime of Algorithm A is $O(N(\log N)^3)$ under the ERH and $O( N(\log N)^4)$ without. 

\begin{lemma}\label{lemma:A}
  Assuming Cram\'er's model, Oppermann's conjecture, and the ERH, Algorithm A can verify that Opperman's conjecture is true for $n\le N$ in $O(N(\log N)^3)$ expected arithmetic operations.  This assumes a suitable table of pseudosquares has been precomputed.  The ERH is used only for the running time, not for correctness.
\end{lemma}

\subsection{Algorithm B}

How much better can we do than an $O(N(\log N)^3)$ running time?
The Atkin-Bernstein prime sieve \cite{AB2004} can find all primes up to $N$ in time $O(N/\log\log N)$ with the use of a wheel, which is only $\log N/\log\log N$ time per prime.  
This implies a heuristic lower bound of $N\log N/\log\log N$ to find $N$ primes.  Because the primes we want to find are spread out, it makes sense to sieve an arithmetic progression.  
Say we choose a modulus $M$ near $N/4\log N$, so that the residue class $1\bmod M$ has about $N^2/M \approx 4N\log N$ integers up to $N^2$.
Applying Cram\'er's model to this arithmetic progression implies we expect to find that around $2N$ of these numbers are prime, since $\log (N^2)=2\log N$.
Now the Atkin-Bernstein sieve takes linear time on an arithmetic progression, since it is not generally possible to deploy a wheel, but by embedding small primes in $M$ we can keep the $\log\log N$ time improvement, thereby matching the desired time bound.

Algorithm B has two major drawbacks.
First, it uses space linear in $N$.  Because we are finding primes in an arithmetic progression, we cannot employ Galway's space improvement \cite{Galway2000}, and so the necessary space will be the square root of the upper bound, hence $N$.  Note that Helfgott's sieve has the same drawback \cite{Helfgott2020}.
Second, although Cram\'er's model tells us that we can expect to find a prime among $2\log n$ integers of size $n^2$, there is no guarantee.  To get the probability of failing to find a prime in such a set down to $1/N$ or less means we need to sample $\gg(\log N)^2$ integers by Lemma \ref{lemma:A}.  Remembering that we put small primes in $M$ can lower this to $(\log N)^2/\log\log N$ by Lemma \ref{lemma:B}, but we now need an overall running time of $O( N(\log N)^2/\log\log N)$.

\begin{lemma}\label{lemma:B}
Assuming Cram\'er's model and Oppermann's conjecture, Algorithm B can verify that Oppermann's conjecture is true for all $n\le N$ with probability $1-o(1)$ in time $O( N(\log N)^2/\log\log N)$ using $O(N)$ space.
\end{lemma}
To get the space use down to $N^c$ for some $c<1$ will incur a higher runtime cost.  This seems to be true even if we use a variant of the sieve of Eratosthenes and choose to allow sieving by non-primes.
See, for example, \cite{Sorenson98}.

\subsection{Algorithm C}
Linear space use is not practical if we want to extend significantly beyond $N=2\times 10^9$.
Our third method is to take Algorithm A, but apply it to an arithmetic progression like we used in Algorithm B but without the sieve to save space.
We choose a prime $R$, with $R>(N^2)^{1/3}=N^{2/3}$ and set our modulus $M=Rm$, where $m$ is composed of small primes such that $M$ is roughly $N/(\log N)^2$ in size.
If we set $a\equiv n^2 \bmod M$, then we would check $n^2+(M-a)+1, n^2+(2M-a)+1,\ldots$ for a prime to search the arithmetic progression $1\bmod M$.
The benefit of this is that we can use the Brillhart-Lehmer-Selfridge (BLS) prime test (see \cite[Theorem 4.1.5]{CP}) with the prime $R$ dividing $\ell-1$ each integer $\ell$ we test for primality; this specialized test is very fast with a running time of $O(\log n)$ arithmetic operations.
If $m$ contains all primes up to, say, $0.2\log n$, the probability of an integer that is $1\bmod M$ being prime is then roughly $\log\log n/(2\log n)$ by Lemma \ref{cramer2}.  This implies an overall running time of $O(N (\log N)^2/\log\log N)$, matching Algorithm B's runtime while using a small fraction of the space.
\begin{lemma}\label{lemma:C}
Assuming Cram\'er's model and Oppermann's conjecture, Algorithm C can verify that Oppermann's conjecture is true for $n\le N$ with probability $1-o(1)$ in time $O( N(\log N)^2/\log\log N)$ using $O(\log N)$ space.
\end{lemma}
In the next section, we describe the final version of our algorithm, which utilizes many of the ideas described in this section.  In particular, we combine Algorithm C with limited sieving to achieve an $O(N\log N \log\log N)$ running time, a mere $(\log\log N)^2$ factor off our heuristic lower bound.

\section{Our approach}

The next new idea is to sieve in small batches.  We choose a parameter $t$ and in one batch, or \textit{segment} as we call it, sieve the arithmetic progression $1\bmod M$ over the range $n^2$ to $(n+t)^2$, hoping to find $2t$ roughly evenly-spaced primes.  We refer to the portion of the segment where we find one prime, either between $(n+i)^2$ and $(n+i)(n+i+1)$ or between this and $(n+i+1)^2$, as an \textit{interval}.  Thus each segment contains $2t$ intervals.

We now give the main steps of the algorithm, with comments to follow.
The algorithm accepts starting values $n$ and $t$, and as we said before, it will verify Oppermann's conjecture over the segment $n^2$ to $(n+t)^2$, thereby finding $2t$ primes.  This is then continued by updating $n:=n+t$ and repeating the process until everything is verified up to $N$.  In the next section, we will pin down a value to use for $t$ to optimize the asymptotic running time, but it will be no larger than a small fractional power of $N$.  We assume here that $t=o(n)$.
\begin{enumerate}
\item \textbf{Set parameters.}
 Let $s$ denote the minimum number of values to check in a given interval.  
 We choose $s$ so that, by Cram\'er's model, one of these candidates is highly likely to be prime.  
 For now, it is safe to assume that $s=\log(n^2)=2\log n$.

  From a precomputed list of primes, choose a prime $R$ with
  $R>(n+t)^{2/3}$ so that $R$ can be used to prove primality for
  potential primes in the current segment that are $1 \bmod R$ using the Brillhart-Lehmer-Selfridge (BLS) prime test mentioned above.
  We then construct a modulus $M$ where $M=Rm$ with $m$ composed of small primes, with $2\mid m$.  We want to find $2t$ primes, so we need
  \[   2ts \approx \frac{(n+t)^2-n^2}{M} \]
  or, after some simplifications, $M\approx (n+t/2)/s$.  This means
  $m\approx (n+t/2)/(Rs)$.  So set $m=2\cdot 3\cdot 5\cdots$ until
  $m$ is the correct size.

\item \textbf{Set up bit vector for sieving.}
  Let $q:=\lceil n^2/M \rceil$ so that $Mq+1$ is the smallest integer
  larger than $n^2$ that is $1\bmod M$.
  Create a bit vector $x$ of length roughly $2ts$, initialized to all zeros,
  where the $i$th bit position in $x$ represents whether $M(q+i)+1$ is \textit{not} prime; so putting all zeros means we assume all of these integers are potentially prime to begin with.

\item \textbf{Sieve by small primes.}
  We sieve $x$ by all primes $p\le B$, with $\gcd(p,M)=1$, where $B$ is determined in the next section.  We want $B$ to be of roughly the same magnitude as the length of $x$, or $2ts$.  
  So $B=o(N)$.

  For each such prime $p$, we compute the first bit position $j$ such
  that $M(q+j)+1$ is divisible by $p$, and then set that bit position and every $p$th bit position after that in $x$, as is done in the sieve of Eratosthenes for arithmetic progressions.

\item \textbf{Look for primes.}
  We repeat this next step $2t$ times, once for each interval: there are two cases (A and B) and we run this for $i=0\ldots (t-1)$ on both cases.

  In case (A), we are looking for a prime between $(n+i)^2$ and $(n+i)(n+i+1)$, and in case (B) we are looking between $(n+i)(n+i+1)$ and $(n+i+1)^2$.  We will describe case (A); case (B) is similar.

  Compute the value of $q_i$, the smallest integer such that $M(q+q_i)+1>(n+i)^2$.  Then search for the first $j$, $j=q_i, q_i+1, q_i+2, \ldots$ so that bit position $j$ in $x$ is still a zero.  Then test
  $M(q+j)+1$ for primality using the BLS test.  If it is prime, we are done with this interval.  If not, continue trying larger values of $j$ until either a prime is found or $M(q+j)+1$ reaches $(n+i)(n+i+1)$.

  We expect to find a prime, but with some small frequency we will fail.  In this case, we have up to 4 other primes $R$ to try, in the style of Algorithm C.  Since sieving is not used, we would resort to some trial division instead.  If this all fails to find a prime (that is, if we exhaust our precomputed list of primes $R$), then we run Algorithm A on the interval.  In practice, this was never necessary.

\end{enumerate}

With the details of our algorithm set, we optimize for $t$, $s$, and $B$ in the next section.

\section{An analysis}

In this section we give the overall heuristic running time of our algorithm from the previous section.  We reiterate that this analysis is heuristic; we assume Cram\'er's model, and we assume Oppermann's conjecture is true.  Also note that the algorithm we analyze is not exactly the same as the code we wrote.  For example, in our code we fixed $s=128$, but below we will choose $s=\log (n+t)$ for the purposes of analysis.

We also make the implicit assumption below that $n\ge \sqrt{N}$, say, so that $n$ is large and we can write $\log(n+t)=O(\log n)$ unambiguously.  Of course the algorithm, as we coded it, works in practice for all $n\le N$.

Our model of computation is an algebraic RAM.  We assume basic arithmetic operations take constant time.

To verify all $n\le N$, our algorithm steps from the previous section are repeated $N/t$ times, once for each segment.  Let us now give the running time for one segment, broken down by the four main steps.

\begin{enumerate}
    \item 
    We set $s$ to the smallest integer larger than $\log (n+t)$.  With $R\approx (n+t)^{2/3}$, we have that $m$ can be as large as $n^{1/3}/\log n$ asymptotically.
    This is large enough that the largest prime in $m$ is $\gg \log (n+t)$.

    The time for this step is dominated by the cost of constructing $m$, since we assume the prime $R$ is precomputed.  Since the largest prime in $m$ is $O(\log n)$, this cost is bounded by finding the primes up to $O(\log n)$, which is $O(\log n\log\log\log n)$ using the sieve of Eratosthenes.
    \item 
    The bit vector $x$ has length $2ts=O(t\log n)$, and the running time for this step is linear in the length of $x$.
    Although $x$ can be reused for different segments, it needs to be cleared to all zeroes each time.
    \item 
    We can find the primes up to $B$ outside the loop over segments so that we pay this cost only once, but asymptotically it does not matter. The cost of sieving is
    \begin{eqnarray*}
        &&  \pi(B) + \sum_{p\le B, \gcd(p,m)=1} \frac{2ts}{p} \\
        &\le& \pi(B) + \sum_{(\log n)/4 < p \le B}\frac{2ts}{p} \\
        &\sim& \pi(B) + (\log\log B - \log\log\log n)(2ts).
    \end{eqnarray*}
    If we assume that $(\log n)/B=o(1)$, this is asymptotically at most
    $B/\log B + 2t\log n\log\log B < B/\log B + 2t\log n\log\log n$.

   Next, we set $B=2ts$, the same as the size of the bit vector $x$,
   so that the running time of this step is $O(t\log n\log\log n)$.

   \item 
   For the running time of this step, we need to first determine the cost of a single BLS prime test.  This cost is a constant
   number of modular exponentiations, or $O(\log n)$, since we can do modular multiplications in constant time (in our computational model).

   We must also calculate the time taken to scan the bit vector $x$ to find zero bits to construct numbers to test for primality.  This is bounded by the number of bits in $x$, or $O(t\log n)$.

   We are searching for exactly $2t$ primes, so the cost of successful prime tests is then bounded by $O(t\log n)$.

   This brings us two questions: \\
   (1) How many failed prime tests are there? and \\
   (2) What guarantee do we have that there will be a successful prime test in each interval?
\end{enumerate}
We will now address these two questions from step 4.

When the algorithm encounters a zero bit in $x$ and then performs a
prime test, under Cram\'er's model, the chance that number is actually
prime is conditional on knowing that it is free of prime divisors
below $B$.  By Lemma \ref{cramer2}, this is $O(\log B/\log n)$
Thus, we would expect to have to test $\log n/\log B$ numbers represented
by zero bits in $x$ before finding the prime we seek.  This answers the first question.

To answer the second question, we simply need to know the chance there is a prime in the interval.
Since $s\sim\log n$, Lemma \ref{cramer2} applies with $v=b$, the bound on primes in $m$.  The probability there are no primes in the interval is $O(1/\log n)$. 

The cost of the last step is expected to be
$O( t \log n (\log n/\log B))$ arithmetic operations, assuming a prime is found.

To match the cost of the sieving step, $O(t\log n \log\log n)$,
we need $\log\log n \gg \log n/\log B$, 
so we set $B=N^{c/\log\log N}$ for a positive constant $c$ like $c=1$.  We can now set $t=B/(2s)$.

With probability $O(1/\log n)$ we failed to find a prime in the interval, but then we use Algorithm C for an expected running time of $O(t \log n/\log\log n)$, which is negligible.  We leave it to the reader to show that the chance Algorithm C fails is exponentially rare making the expected cost of a use of Algorithm A negligible as well.

We multiply the cost of one segment by $N/t$ to obtain the overall running time of $O(N\log N \log\log N)$ using $N^{O(1/\log\log N)}$ space.

\begin{theorem}
Under the assumption of Cram\'er's model for the runtime analysis only, if Opperman's conjecture is true, our algorithm from Section 4 will verify Opperman's conjecture for all $n\le N$ using $O(N\log N\log\log N)$ arithmetic operations and at most $N^{O(1/\log\log N)}$ space.
\end{theorem}


\section{Implementation Details}

We conclude with a discussion of some of our implementation details and comments on our computation.

For each of the four main steps of the algorithm, we have a few comments.
\begin{enumerate}
\item 
In practice, we used $s=128$.  We adapted the value of $t$ each segment; our bit vector $x$ was of a fixed size, so we chose $M$ so that $t$ would
fall between $256$ and $512$.  The average value of $t$ was around $450$.
This all implies that the length of the bit vector $x$ and the value of $B$ were fixed at $2^{17}=131072$, which easily fit in cache.  We used the C++ \texttt{bitset} class for $x$ and the masks (see below).

We used a Sage script to precompute a 2-dimensional list of integers that are provably prime to serve as $R$ values.  Five primes of each size are found, and the size difference between adjacent groups of five is a multiple around 1.25.  The list has
a total of around 50 primes to cover the scope of our computation.

In practice we mostly used $m=2\cdot 3\cdot 5$ with a small possible additional factor to make $M$ fit as we described.

\item
  Having bit values of 1 mean composite and 0 possibly prime is opposite
  what is normally done in practice, but this works better for us as you will see in the next comment.
\item 
  In addition to the sieving described above, for each of the primes $p<64$ we
  created a mask the same data type and length as $x$, where all bits of the mask
  are zero except at bit positions that are multiples of $p$, which are
  set to 1.  Then, to sieve by a small prime $p$, we first shift the mask to line up with the start point $n^2$, and then perform a bitwise or of the mask with $x$.  This should greatly speed sieving by small primes, since $x$ is packed 64 bits to a word.

  If we were to use a larger value for $t$ so that $x$ was much larger, the space used by the masks might use up cache and make the program slower.  Finding the right balance for our hardware took some trial and error.

  This optimization shows some improvement in practice, but should have
  no effect on the asymptotic running time.

\item 
  To perform the BLS prime test we must construct the value
    of $M(q+j)+1$, which fits in a 128-bit integer.  However, the test itself requires modular exponentiation.  We found that converting this 128-bit integer into the \texttt{mpz\_t} integer data type in GMP and using GMP's modular exponentiation function worked the best.  Our own hand-coded modular exponentiation routine, using Montgomery multiplication, was not as fast.
    Also, we always used the base 2 for the prime test; if that failed, rather than trying another base we simply looked for another $j$ to try.

    The sieving was quite successful in that the average number of failed prime tests per interval was only around 2.

  In practice, the first prime $R$ we used worked almost always.  About one time in $10^4$ the second $R$ value was used with Algorithm C.  It never went to the third $R$ value during our computation, and we never had to resort to Algorithm A.
  
\end{enumerate}

Aside from the Sage script to precompute the primes used for $R$ values, all our code was written in C++.  We used the GMP library only for prime testing as described above, and we use MPI to parallelize the code.
Parallelization was easy - we simply striped on large groups of segments.  We used groups of segments rather than individual segments so that the value of $t$ could be adjusted on each segment to maximally utilize the fixed-sized bit vector $x$.

See \url{https://github.com/sorenson64/olc} for our code and most of our data.

\section*{Acknowledgements}

We thank Frank Levinson for supporting Butler University's computing infrastructure, and both authors were supported in part by grants from the Butler Holcomb Awards Committee.

\nocite{GnuMP}

\bibliographystyle{amsplain}

\providecommand{\bysame}{\leavevmode\hbox to3em{\hrulefill}\thinspace}
\providecommand{\MR}{\relax\ifhmode\unskip\space\fi MR }
\providecommand{\MRhref}[2]{%
  \href{http://www.ams.org/mathscinet-getitem?mr=#1}{#2}
}
\providecommand{\href}[2]{#2}

\end{document}